\documentclass[conference]{IEEEtran}
\IEEEoverridecommandlockouts
\usepackage{cite}
\usepackage{amsmath,amssymb,amsfonts}
\usepackage{algorithmic}
\usepackage{graphicx}
\usepackage{textcomp}
\usepackage{xcolor}
\def\BibTeX{{\rm B\kern-.05em{\sc i\kern-.025em b}\kern-.08em
    T\kern-.1667em\lower.7ex\hbox{E}\kern-.125emX}}
\begin{document}

\title{Estimation of Strong Structural Controllable Subspace of Network: Equitable Partition Method
}

\author{\IEEEauthorblockN{Lanhao Zhao}
\IEEEauthorblockA{
\textit{College of Automation, Qingdao University}\\
Qingdao, China \\
lugh56.2007@163.com}
}

\maketitle
\newtheorem{theorem}{Theorem}
\newtheorem{lemma}{Lemma}
\newtheorem{rmk}{Remark}
\newenvironment{proof}{{\noindent\it Proof}\quad}{\hfill $\square$\par}
\newtheorem{exm}{Example}
\newtheorem{de}{Definition}
\newtheorem{co}{Corollary}
\newtheorem{pro}{Proposition}

\begin{abstract}
In this paper, the strong structural controllability of the network is analyzed. Based on the unified definition of equitable partition for kinds of scene, the upper bound of the strong structural controllable subspace in different scenarios is given, and the strong structural observability is analyzed by using the characteristics of the dual system. Finally, the practical significance when the dimension of the strong structural controllable subspace is less than the number of individuals is given, and an invariant attribute of strong structural controllability analysis is proposed.
\end{abstract}


\section{Introduction}
Strong structural controllability \cite{1} represents the controllability attribute of the network under the condition of ignoring the weight selection, which more deeply reflects the structural attribute of the network.

In recent years, some results of strong structural controllability of networks have been obtained \cite{2,3,4}. The existing results largely discuss whether a special structure is strongly structural controllable. These results show that many network topologies are not strongly structural controllable, so the problem is how to measure the degree that these network topologies are close to strongly structural controllable.

The estimation of strongly structural controllable subspaces is an attempt to deal with this problem. We call it SSC index here, which is equal to the dimension of strong structural controllable subspace. Some estimation methods are obtained by using algebraic properties, zero forcing set, distance partition and other tools \cite{5,6,7,8}.

The results based on equitable partition can not be ignored when it comes to the estimation of the controllable subspace of the system. Similar results are obtained for all positive networks, signed networks, matrix weighted networks, time-delay networks, signed matrix weighted networks and heterogeneous networks based on the definition of equivalent partition in different scenarios \cite{9,10,11,12,13,14}. One idea is whether we can use the concept of equitable partition to estimate the strongly structural controllable subspace based on these existing results?

In this paper, the strong structural controllability of the network is analyzed. Based on the definitions of various equitable partitions given in our previous work, the upper bound of the network strong structure controllable subspace in different scenarios is given, and the strong structural observability is analyzed by using the characteristics of the dual system. Finally, the practical significance when the dimension of the strong structural controllable subspace is less than the number of individuals is given, and an invariant attribute in the strong structure controllability analysis is proposed.

\section{Notations and Preliminaries}
In this paper, $R$ stands for the set of real number, $I_{m}$ and $0_{m}$ denote identity matrix and zero matrix with dimension $m$, respectively. $G=\{V, E, \mathcal{A}\}$ is called as matrix weighted graph, in which $V=\left\{v_{1}, v_{2} \cdots v_{n}\right\}$ represents the vertex set.  $\mathcal{A}=\left[\mathcal{A}_{i j}\right] \in R^{nd \times nd}$ is the weighted adjacency matrix, where $\mathcal{A}_{i j}$ belonging to $\mathbb{R}^{d \times d}$.  $E$ represents the edge set. The set of neighbors of agent $v_{i}$ is represented by $N_{i}$. $d_{i}=\sum_{j\in{N}_{i}}\mathcal{A}_{i j}$ denotes the degree of $i$ for matrix-weighted signed graph. Let $L=D-\mathcal{A}$ be the Laplacian matrix of $G$, where $D=\operatorname{diag}\left(d_{1}, \cdots, d_{n}\right)$. The entries of $L$ can be writen as follows
$$
l_{i j}=\left\{\begin{array}{ll}
	d_{i}, & i=j \\
	-\mathcal{A}_{i j}, & i \neq j
\end{array}\right . 
$$
Graph partition: For the vertex set $V$ of a graph, its subset $V_{i}$ is called a cell. It is called a trivial cell  if the cell contains only one vertex, otherwise it is a nontrivial cell. If any vertex in $V_{1}$ also belongs to $V_{2}$,
then $V_{1}$ is a sub-cell of $V_{2}$. We define $\pi=\left\{V_{1}, V_{2}, \cdots, V_{k}\right\}$. Then $\pi$ is a partition of graph when $V_{i} \cap V_{j}=\varnothing$ and $\bigcup_{i} V_{i}=V$ for $0<i, j<k$ and $i \neq j$. The characteristic matrix is
$$
P_{i j} \triangleq\left\{\begin{array}{ll}
	I_{d\times d} , & i \in V_{j} \\
	0_{d\times d} , & i \notin V_{j}
\end{array}\right.
$$
\begin{exm}
	The characteristic matrix of $\pi=\{\{1,2\},\{3,4,5\}\}$ is
	$$
	P(\pi)=\left[\begin{array}{ll}
		I_{d\times d} & 0_{d\times d} \\
		I_{d\times d} & 0_{d\times d} \\
		0_{d\times d} & I_{d\times d} \\
		0_{d\times d} & I_{d\times d} \\
		0_{d\times d} & I_{d\times d}
	\end{array}\right].
	$$
\end{exm} 
\section{Networks with Frist order Dynamics}
Consider a multiagent system with $n$ agents which state is denoted by the $x_{i}(t)\in \mathbb{R}^{d}$. The leader and follower are distinguished according to whether the agent receives external input signals. We assume that the first $m$ $(m<n)$ individuals are named as leaders where $V_{L}=\left\{v_{1}, \cdots, v_{m}\right\}$ is the leader set, and the follower set is represented by $V_{F}=\left\{v_{m+1}, \cdots, v_{n}\right\} $.

For the matrix-weighted signed networks under general linear dynamics, all followers are governed by the following dynamics
$$
\dot{x}_{i}(t)= x_{i}(t)+ u_{i}(t).
$$

For every leader, its dynamics is
$$
\dot{x}_{i}(t)= x_{i}(t)+ u_{i}(t)+ y_{i}(t)
$$
where $u_{i} \in R^{p}$ reflects the influence that each individual receives from the others, $y_{l} \in R^{q}$ represents an external input signal. The update rules based on neighbors are as follows
$$
u_{i}(t)=\sum_{j \in N_{i}} \left[\mathcal{A}_{i j} x_{j}(t)-\mathcal{A}_{i j} x_{i}(t)\right]
$$
where $K$ stands for the feedback gain, $\mathcal{A}_{i j}$ is the connection weight between the individual $i$ and $j$. Denote $x(k)=\left[x_{1}(t), \cdots, x_{n}(t)\right]^{T}$ as the aggregate state vector and $y(k)=\left[y_{1}(t), \cdots, y_{l}(t)\right]^{T}$ as the control input vector. Then the expression of multiagent system can be written as
$$
\dot{x}(t)=L  x(t)+ M u(t). \quad\quad\quad  (1)
$$
where matrix $M$ is used to distinguish the leaders from the followers.
$$
M_{i l}=\left\{\begin{array}{ll}
	I_{d\times d} & i=l \\
	0_{d\times d} & \text { otherwise }
\end{array}\right.
$$
\begin{rmk}
Considering that scalar weight network is a special form of matrix weight network \cite{11}, system (1) has a wide coverage of different scenarios.
\end{rmk}

\section{Main Results}
\subsection{Strongly structural controllable subspace and Equitable Partition}
For system (1), the system is controllable if and only if the matrix $\left[M \quad L M \quad L^{2} M \cdots {L}^{dn-1} M\right]$ has full row rank, Suppose the i-th weight selection method is adopted and the controllable subspace of system (1) is
$$
\mathcal{W}_{i}=\left\langle L \mid M\right\rangle=i m(M)+L \times i m(M)+\cdots+L^{dn-1} \times i m(M).
$$
It is a minimal $L-$ invariant subspace containing ${im}(M)$ \cite{13}.

A network with $V$ leaders is strong structural controllable if and only if (L, M) is a controllable pair for any choice of weight, or in other words,
At the same time, the dimension of strong structurally controllable
subspace (SSCS) is
$$
\mathcal{W}'= min (\mathcal{W}_{i})
$$
It is obvious that
\begin{lemma}
	If for any $\mathcal{W}_{j}$, $\mathcal{W}_{j} \subseteq \mathcal{W}_{*} $ hold, then $\mathcal{W}' \subseteq \mathcal{W}_{*} $
\end{lemma}
\begin{de}
	Denote a matrix-weighted graph as $G$, and let $\pi=\left\{V_{1}, V_{2}, \ldots, V_{k}\right\}$ be a partition of $G$. The partition $\pi$ is said to be an equitable partition (EP) of $G$ if for any $r, s \in V_{i}, i, j=1,2, \ldots, k$
	$$\sum_{t_{1} \in V_{j}, t_{1} \in {N}_{r}} \mathcal{A}_{r t_{1}}=\sum_{t_{2} \in V_{j}, t_{2} \in {N}_{s}} \mathcal{A}_{s t_{2}}$$
	where $\left(t_{1}, r\right),\left(t_{2}, s\right) \in E$. 	
\end{de}
\begin{rmk}
   The definition can be applied to a variety of scenarios with some assumptions, For example, Time-delay system \cite{12}, Heterogeneous system \cite{14}.
\end{rmk}
\begin{rmk}
We note that the matrix weight network often assumes that the weight matrix is symmetric for the convenience of consensus research, and Definition 1 can also deal with the case of asymmetric matrix.
\end{rmk}
\subsection{Upper bound of strongly structural controllable subspace}
Denote $d(v_{i},Q)=\sum_{v_{j}\in Q}\mathcal{A}_{i j}$ and $d(V_{i},Q)=d(v,Q)$ for all $v \in V_{i}$, then we give the concept of quotient graph
\begin{de}
	For an equitable partition $\pi=$ $\left\{V_{1}, V_{2}, \ldots, V_{s}\right\}$ of a matrix-weighted network $G$, the quotient graph of $G$ over $\pi$ is a matrix-weighted network denoted by $G/\pi$ with the node set $V(G/\pi)=\left\{v_{1}, v_{2}, \ldots, v_{s}\right\}$ and the edge set is $E(G/ \pi)=\left\{\left(v_{i}, v_{j}\right) \mid d\left(V_{i}, V_{j}\right) \neq 0_{d \times d}\right\},$ where
	the weight of edge $\left(V_{i}, V_{j}\right)$ is $d\left(V_{i}, V_{j}\right)$ for $i \neq j \in \underline{s}$.
\end{de}

Denote $L_{\pi}$ as the Laplacian matrix of $G/\pi$
$$
\left(L_{\pi}\right)_{i j}=\left\{\begin{array}{ll}
	\sum_{v_{j} \in V(G/\pi)} d\left(V_{i}, V_{j}\right), & i=j \\
	-d\left(V_{i}, V_{j}\right), & i \neq j
\end{array}\right.
$$
\begin{lemma}
	$ \pi=\left\{V_{1}, V_{2}, \cdots, V_{s}\right\}$ is an EP for matrix-weighted signed graph $G$ and $P_{\pi}$ is the characteristic matrix. Then $L$ satisfies
	$$
	L P_{\pi}=P_{\pi} L_{\pi}.
	$$
	Furthermore $i m\left(P_{\pi}\right)$ is $L-$ invariant. 	
\end{lemma}
\begin{proof}
	The proof is similar to Lemma 1 in \cite{11}, and thus is omitted. In particular, there are similar conclusions for time-delay systems and heterogeneous systems \cite{12,14}.
\end{proof}

Next, we try to estimate the strongly structural controllable subspace. From definition 1, the existence of equitable partition depends on the selection of weight. Obviously, for each cell contains one and only one node, this partition always meets the definition, and we do not consider this trival scenario. Consider the following scenario, there is non trivial equitable parition under some weight selection methods, and other forms of weight selection methods do not meet definition 1. As shown in the Fig 1. There is non trivial equitable parition only if $A12=A13,A24=A34$. 

\begin{figure}
	\centering
	\includegraphics[scale=.8]{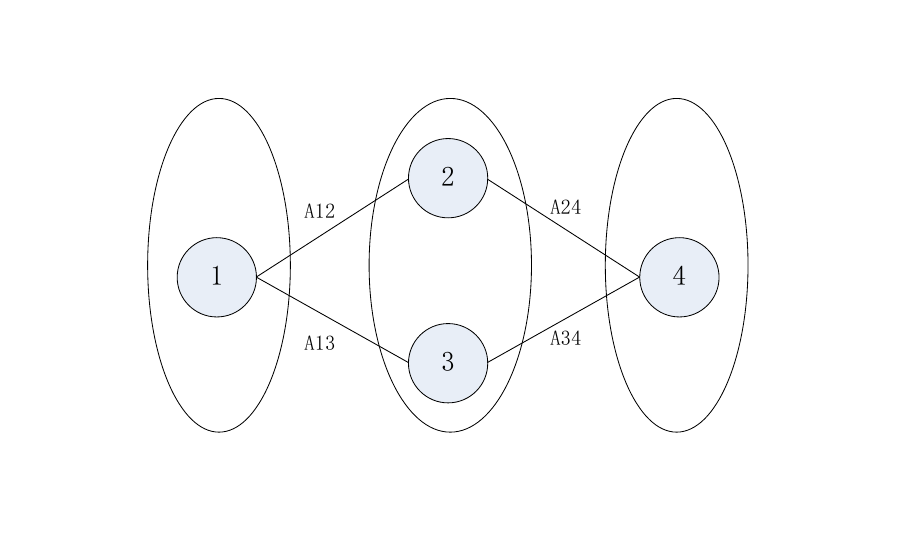}
	\caption{A matrix weighted signed graph with fixed topology including nontrival cell}
	\label{FIG:1}
\end{figure}

For convenience, for the same network, we use $\mathcal{W}_{jn}$ to represent the controllable subspace of the network with nontrivial equitable partition by appropriate weight, and $\mathcal{W}_{in}$ to represent the controllable subspace of the network without nontrivial equivalent partition by appropriate weight.

For any $\mathcal{W}_{jn}$, we can get
\begin{lemma}
	The controllable subspace $\mathcal{W}_{jn}$ satisfies $\mathcal{W}_{jn} \subseteq i m\left(P_{\pi j}\right)$.
\end{lemma}
\begin{proof}
	Every column of $M$ is also a column of $P_{\pi j}$, and then every column of $M$ is also a column of $P_{\pi j}$. It follows that $ im(M) \subseteq im\left(P_{\pi j}\right)$. By Lemma $2$,  ${im}(P_{\pi j})$ is
	$\tilde{L}-$ invariant. Then
	$
	\begin{aligned}
		\mathcal{W} &=im(M)+L \times i m(M)+\cdots+L^{dn-1} \times i m(M) \\
		& \subseteq im\left(P_{\pi j}\right)+L \times i m\left(P_{\pi j}\right)+\cdots+L^{dn-1} \times im\left(P_{\pi j}\right) \\
		&=im\left(P_{\pi j}\right).
	\end{aligned}
	$
\end{proof}
From Lemma 3, the dimension of controllable subspace is related to the number of cells in the partition. In order to obtain a more accurate estimation, we choose the equitable partition $\pi j m$ with the least cells under the arbitrary selection of weight.
\begin{theorem}
	The strong structural controllable subspace $\mathcal{W}'$ satisfies $\mathcal{W}' \subseteq i m\left(P_{\pi j m}\right)$.
\end{theorem}
\begin{proof}
	Form lemma 3, $\mathcal{W}_{jn}  \subseteq i m\left(P_{\pi j m}\right)  \subseteq i m\left(P_{\pi j}\right) $.
	Then form Lemma 1, if for any $\mathcal{W}_{j}$, $\mathcal{W}_{j} \subseteq \mathcal{W}_{*} $ hold, then $\mathcal{W}' \subseteq \mathcal{W}_{*} $, thus $\mathcal{W}' \subseteq i m\left(P_{\pi j m}\right)$.
\end{proof}
\begin{rmk}
	Theorem 1 gives the upper bound of strong structural controllable subspace. According to the \cite{12,14}, under appropriate assumptions, we can get a conclusion similar to theorem 1 for tiem-delay networks, heterogeneous networks and asymmetric matrix weighted networks. It shows that similar to the treatment method of controllable subspace, the concept of equitable partition can be used to deal with the estimation problem of strong structure controllable subspace too. With the estimation of the lower bound of strong structure controllable subspace given in previous work\cite{7}, a more accurate characterization of strong structure controllable subspace can be oatained.
\end{rmk}
\subsection{Strongly structural observability}

Compared with controllability, observability is also a topic worthy of discussion. It is used to measure the ability to reconstruct the whole network state. The system can be described as
$$
\dot{x}(t)=L x(t)+M y(t)
\quad\dot{y}(t)=M^{T} x(t) \quad\quad\quad\quad (2)
$$

The observability of system (2) is equivalent to the controllability of its dual system. The dual system is represented as follows.
$$
\dot{x}(t)=L^{T} x(t)+M y(t),
\quad\dot{y}(t)=M^{T} x(t) \quad\quad\quad\quad (3)
$$

And then we transform the observability problem of system (2) into the controllability problem of system (3). 

For scalar weighted undirected networks, $L=L^{T}$, system (1) is the same as system (3).
For scalar weighted directed networks, if the system (1) takes the opposite direction in all edge in network,then system (3) can be oatained.
For matrix weight networks, the situation is similar when the weight matrix is symmetric and undirected.
For matrix weight networks, when the weight matrix is asymmetric matrix or directed graph, it needs special discussion.

Through the above analysis, we transform the algebraic relationship between the dual system and the original system into the relationship on the graph, which provides a convenient perspective for analyzing the controllability and observability of the network. It is also applicable to strong structure controllability and strong structure observability.

\subsection{Invariant attribute of strong structural controllability analysis}
Consider the following scenario: the estimated value of the upper bound is less than the number of nodes in the network. The direct conclusion is that the network is not strongly structurally controllable. In addition, this estimate value also shows the maximum possible number of controllable nodes in the network under any weight selection.

Similar to the estimation of controllable subspace, on the one hand, this estimate value reflects the degree to which the system is close to strong structure controllability. On the other hand, when the upper bound of the estimation is the dimension of strong structure controllable subspace, this estimate value shows that there are always some nodes that are controllable for this type of network no matter how the weight is selected, these nodes are not affected by the weight selection. This reflects an invariant property of the network.
\section{Conclusion}
In this paper, the strong structural controllability of the network is analyzed. Based on the definitions of various equitable partitions given in our previous work, the upper bound of the network strong structure controllable subspace in different scenarios is given, and the strong structural observability is analyzed by using the characteristics of the dual system. Finally, the practical significance when the dimension of the strong structural controllable subspace is less than the number of individuals is given, and an invariant attribute of strong structural controllability analysis is proposed.

The application of equitable partition in dealing with strong structural controllable subspace also enlightens us that we should comprehensively consider the existing results on controllability and seek new ideas for reference. In particular, for the estimation of controllable subspace of network, in addition to equitable partition method and distance partition method, there are also estimation methods relying on a variety of controllable structures. Can these methods be used to estimate strong structure controllable subspace? It is an interesting topic.

\bibliographystyle{IEEEtran}
\bibliography{ref}

\begin{thebibliography}{10}
\providecommand{\url}[1]{#1}
\csname url@samestyle\endcsname
\providecommand{\newblock}{\relax}
\providecommand{\bibinfo}[2]{#2}
\providecommand{\BIBentrySTDinterwordspacing}{\spaceskip=0pt\relax}
\providecommand{\BIBentryALTinterwordstretchfactor}{4}
\providecommand{\BIBentryALTinterwordspacing}{\spaceskip=\fontdimen2\font plus
\BIBentryALTinterwordstretchfactor\fontdimen3\font minus
  \fontdimen4\font\relax}
\providecommand{\BIBforeignlanguage}[2]{{%
\expandafter\ifx\csname l@#1\endcsname\relax
\typeout{** WARNING: IEEEtran.bst: No hyphenation pattern has been}%
\typeout{** loaded for the language `#1'. Using the pattern for}%
\typeout{** the default language instead.}%
\else
\language=\csname l@#1\endcsname
\fi
#2}}
\providecommand{\BIBdecl}{\relax}
\BIBdecl

\bibitem{1}
H.~Mayeda and T.~Yamada, ``Strong structural controllability,'' \emph{SIAM
  Journal on Control and Optimization}, vol.~17, no.~1, pp. 123--138, 1979.

\bibitem{2}
A.~Chapman and M.~Mesbahi, ``On strong structural controllability of networked
  systems: A constrained matching approach,'' in \emph{2013 American Control
  Conference}.\hskip 1em plus 0.5em minus 0.4em\relax IEEE, 2013, pp.
  6126--6131.

\bibitem{3}
S.~S. Mousavi, M.~Haeri, and M.~Mesbahi, ``On the structural and strong
  structural controllability of undirected networks,'' \emph{IEEE Transactions
  on Automatic Control}, vol.~63, no.~7, pp. 2234--2241, 2017.

\bibitem{4}
J.~Jia, H.~J. van Waarde, H.~L. Trentelman, and M.~K. Camlibel, ``A unifying
  framework for strong structural controllability,'' \emph{IEEE Transactions on
  Automatic Control}, vol.~66, no.~1, pp. 391--398, 2020.

\bibitem{5}
J.~C. Jarczyk, F.~Svaricek, and B.~Alt, ``Determination of the dimensions of
  strong structural controllable subspaces,'' \emph{IFAC Proceedings Volumes},
  vol.~43, no.~21, pp. 131--137, 2010.

\bibitem{6}
A.~Chapman and M.~Mesbahi, ``On strong structural controllability of networked
  systems: A constrained matching approach,'' in \emph{2013 American Control
  Conference}.\hskip 1em plus 0.5em minus 0.4em\relax IEEE, 2013, pp.
  6126--6131.

\bibitem{7}
M.~Shabbir, W.~Abbas, and Y.~Yaz{\i}c{\i}o{\u{g}}lu, ``On the computation of
  the distance-based lower bound on strong structural controllability in
  networks,'' in \emph{2019 IEEE 58th Conference on Decision and Control
  (CDC)}.\hskip 1em plus 0.5em minus 0.4em\relax IEEE, 2019, pp. 5468--5473.

\bibitem{8}
Y.~Yaz{\i}c{\i}o{\u{g}}lu, M.~Shabbir, W.~Abbas, and X.~Koutsoukos, ``Strong
  structural controllability of diffusively coupled networks: Comparison of
  bounds based on distances and zero forcing,'' in \emph{2020 59th IEEE
  Conference on Decision and Control (CDC)}.\hskip 1em plus 0.5em minus
  0.4em\relax IEEE, 2020, pp. 566--571.

\bibitem{9}
S.~Zhang, M.~Cao, and M.~K. Camlibel, ``Upper and lower bounds for controllable
  subspaces of networks of diffusively coupled agents,'' \emph{IEEE
  Transactions on Automatic control}, vol.~59, no.~3, pp. 745--750, 2013.

\bibitem{10}
C.~Sun, G.~Hu, and L.~Xie, ``Controllability of multiagent networks with
  antagonistic interactions,'' \emph{IEEE transactions on automatic control},
  vol.~62, no.~10, pp. 5457--5462, 2017.

\bibitem{11}
L.~Pan, H.~Shao, M.~Mesbahi, Y.~Xi, and D.~Li, ``On the controllability of
  matrix-weighted networks,'' \emph{IEEE Control Systems Letters}, vol.~4,
  no.~3, pp. 572--577, 2020.

\bibitem{12}
L.~Zhao, Z.~Ji, Y.~Liu, and C.~Lin, ``Controllability of general linear
  discrete multiagent systems with directed and weighted signed networks,''
  \emph{Journal of Systems Science and Complexity}, 2022.

\bibitem{13}
H.~Gao, Z.~Ji, and T.~Hou, ``Equitable partitions in the controllability of
  undirected signed graphs,'' in \emph{2018 IEEE 14th International Conference
  on Control and Automation (ICCA)}.\hskip 1em plus 0.5em minus 0.4em\relax
  IEEE, 2018, pp. 532--537.

\bibitem{14}
L.~Zhao, Z.~Ji, Y.~Liu, and C.~Lin, ``Controllability and observability of
  linear multi-agent systems over matrix-weighted signed networks,''
  \emph{arXiv.2204.00995}, 2022.

\end{thebibliography}

\end{document}